\documentclass[12pt,a4paper]{amsart}

\usepackage{amsmath,amssymb,amsthm}
\usepackage{mathtools}
\usepackage{enumitem}
\usepackage{xcolor}
\usepackage{hyperref}
\hypersetup{
  colorlinks=true,
  linkcolor={blue!60!black},
  citecolor={blue!60!black},
  urlcolor={blue!60!black}
}

\newtheorem{theorem}{Theorem}[section]
\newtheorem{lemma}[theorem]{Lemma}
\newtheorem{proposition}[theorem]{Proposition}
\newtheorem{corollary}[theorem]{Corollary}

\theoremstyle{definition}
\newtheorem{definition}[theorem]{Definition}
\newtheorem{example}[theorem]{Example}
\newtheorem{remark}[theorem]{Remark}

\newcommand{\bul}{\bullet}
\newcommand{\ld}{\cdot}
\newcommand{\lam}{\lambda}
\newcommand{\Aut}{\operatorname{Aut}}
\newcommand{\Fix}{\operatorname{Fix}}
\newcommand{\Ker}{\ker}

\newcommand{\raisedstar}{\texorpdfstring{\raisebox{.2ex}{$\star$}}{*}}

\title{Associativity of the \raisedstar-product of skew braces}
\author{Andrea Sciandra}
\date{}
\address{%
\parbox[b]{0.9\linewidth}{Département de Mathématiques, Université Libre de Bruxelles, Boulevard du Triomphe, B-1050
 Bruxelles, Belgium.}}
\email{andrea.sciandra@ulb.be}
\urladdr{\url{www.andreasciandra.com}}



\begin{document}

\begin{abstract}
For braces, associativity of the $\star$-product is equivalent to two-sidedness. We extend this result to skew braces by proving that any two of the following conditions imply the third: two-sidedness, associativity of the $\star$-product, and centrality of $A\star A$ in the additive group. We then study two-sided skew braces with associative $\star$-product through the radical ring associated with their additive abelianization. This yields a canonical fibre-product description modulo an annihilator ideal, exact sequences relating higher $\star$-products to the powers of the radical ring, coincidence of the left, right and strong $\star$-series, and sharp bounds on their nilpotency indices. We also obtain an explicit description of the lower central series, classify the structures whose additive abelianization is infinite cyclic, and characterize when they are bi-skew.
\end{abstract}

\maketitle

\tableofcontents

\section{Introduction}
It is known that the $\star$-product of a brace provides a natural link with radical rings: associativity of this operation is equivalent to two-sidedness \cite{Rump,Lau}. For skew braces the situation is substantially different. The additive group need not be abelian, and in general neither associativity of $\star$ nor two-sidedness implies the other. The main result of this paper (Theorem \ref{thm:extension}) identifies precisely the additional condition which restores the classical equivalence. More precisely, we prove that, for a skew brace $(A,\cdot,\bullet)$, any two of the following conditions imply the third:
\[
A\text{ is two-sided},\qquad
\star\text{ is associative},\qquad
A\star A\subseteq Z(A^\cdot).
\]
Thus, centrality of the subgroup generated by the $\star$-products is exactly the missing ingredient in passing from braces to skew braces.

A key step in the proof is the analysis of almost trivial skew braces. If $A=(G,\cdot,\cdot^{\mathrm{op}})$, then $a\star b=[a,b^{-1}]$ and $\star$ is associative if and only if $G$ is nilpotent of class at most two (Proposition \ref{prop:quasi}). This also gives simple examples showing that two-sidedness alone does not force associativity. In the opposite direction, we give a general construction, starting from an index-two subgroup and an involutive automorphism, which produces skew braces with associative $\star$-product that need not be two-sided (Proposition \ref{prop:index-two}). In particular, the three conditions in the main theorem genuinely capture phenomena which are absent in the classical brace setting. 

Section \ref{sec:applications} is devoted to applications for two-sided skew braces with associative $\star$-product. In this setting $A\star A\subseteq Z(A^\cdot)$, so the $\star$-product is distributive in both variables with respect to the additive group operation. Then $D=[A^\cdot,A^\cdot]$ is annihilated on both sides by $\star$, and the quotient $R=A/D$ inherits a canonical Jacobson radical ring structure. We obtain a structural description of $A$ modulo an annihilator ideal: writing $B=A\star A$ and $I=B\cap D$, we obtain a canonical fibre-product skew brace decomposition 
\[
A/I\cong A/B\times_{A/(BD)}R,
\]
with $I\subseteq\operatorname{Ann}(A)$ (Proposition \ref{app:fibre-product}). We then study the subgroups $P_n(A)$ generated by $\star$-products of length $n$. They coincide with the left, right and strong $\star$-series and fit into exact sequences relating them to the powers $R^n$ (Proposition \ref{prop:powers}), which yield sharp bounds on the corresponding nilpotency indices. We also obtain an explicit formula for the lower central series,
\[
\Gamma_n(A)=\gamma_n(A^\cdot)\cdot P_n(A)\]
which separates the group-theoretic and ring-theoretic contributions to central nilpotence (Theorem \ref{app:lower-central}). Finally, when $A^\cdot/[A^\cdot,A^\cdot]\cong\mathbb Z$, we classify all two-sided skew brace structures with associative $\star$-product in terms of the subgroup $Z(A^\cdot)\cap[A^\cdot,A^\cdot]$ (Proposition \ref{app:cyclic-abelianization}). These structures satisfy $P_3(A)={1}$, a condition which, under the standing assumptions of the section, is proven to be equivalent to $A$ being bi-skew (Corollary \ref{app:bi-skew}).

\section{Preliminaries}

First, we recall some preliminary notions and results that will be useful in the following.

\begin{definition}[{\cite{GV}}]
A \emph{skew brace} is a triple $(A, \ld, \bul)$ where $A^{\ld}:=(A,\ld,1)$ and
$A^{\bul}:=(A,\bul,1)$ are both groups such that
\begin{equation}\label{eq:brace-axiom}
  a \bul (b \ld c) \;=\; (a \bul b) \ld a^{-1} \ld (a \bul c)
  \qquad \text{for all}\, a,b,c \in A,
\end{equation}
where $a^{-1}$ denotes the inverse of $a$ in $A^{\ld}$. The inverse of $a$ in $A^{\bul}$ will be denoted by $\overline{a}$.
\end{definition}
A skew brace $(A,\cdot,\bullet)$ is a \textit{brace} if $A^{\ld}$ is abelian \cite{Rump}. A skew brace is \emph{two-sided} if in addition
\begin{equation}\label{eq:twosided}
(a \ld b) \bul c = (a \bul c)\cdot c^{-1} \ld (b \bul c)
  \qquad \text{for all}\ a,b,c \in A.
\end{equation}
It is known that two-sided braces are in
bijective correspondence with radical rings \cite{Rump2}. Recall that a \textit{radical ring} is a ring $(R,+,\cdot)$ without identity such that for each $x\in R$ there exists $y\in R$ such that $x+y+xy=0$. Given a radical ring $(R,+,\cdot)$, one has that $(R,+,\bullet)$ is a two-sided brace where $a\bullet b:=a+ab+b$, for all $a,b\in R$. Conversely, if $(A,\cdot,\bullet)$ is a two-sided brace, the
operation $a\star b:=a^{-1}\cdot(a\bullet b)\cdot b^{-1}$, for all $a,b\in A$, makes $(A,\cdot,\star)$ into a radical ring.

We also recall that, given a skew brace $(A,\ld,\bul)$, there is a morphism of groups 
\[
\lambda:A^{\bul}\to\mathrm{Aut}(A^{\ld}),\quad a\longmapsto a^{-1}\cdot(a \bul b)
\]
as proven in \cite[Corollary 1.10]{GV}. In particular, $\lam_1=\mathrm{id}$ and $\lam_{a \bul b} = \lam_a\lam_b.$
For an arbitrary skew brace
$(A,\ld,\bul)$, one defines the $\star$-\textit{product} as above
\begin{equation}\label{eq:star}
  a \star b \;:=\; a^{-1} \ld (a \bul b) \ld b^{-1} \;=\; \lam_a(b) \ld b^{-1}.
\end{equation}
Thus, the operation $\star$ measures the distance between $\ld$ and $\bul$: $a \bul b \;=\; a \ld (a \star b) \ld b$. Observe that the identity element $1$ is a two-sided absorbing element for $\star$: indeed
$a \star 1 = \lam_a(1) \ld 1^{-1} = 1$ and $1 \star a = \lam_1(a)\ld a^{-1} = a \ld a^{-1} = 1$. The operation $\star$ has no identity element in general, as already observed in the brace setting.

From now on, given subsets $X,Y\subseteq A$, we set
\[
X\star Y:=\langle x\star y\mid x\in X,\ y\in Y\rangle_{\cdot},
\]
i.e.\ $X\star Y$ is the subgroup generated in $A^{\cdot}$ by elements $x\star y$, with $x\in X$ and $y\in Y$. An \emph{ideal} $I$ of a skew brace $(A,\cdot,\bullet)$ is a subgroup which is normal in both
underlying groups and such that $\lambda_a(I)\subseteq I$, for all $a\in A$. We use
\[
\Fix(\lambda):=\{x\in A\mid\lambda_a(x)=x\text{ for every }a\in A\},
\]
and write
$\operatorname{Ann}(A):=Z(A^{\cdot})\cap\Ker(\lambda)\cap\Fix(\lambda)$ for the \textit{annihilator} of the skew brace, which also coincides with $Z(A^{\cdot})\cap\mathrm{ker}(\lambda)\cap Z(A^{\bul})$.

The \emph{opposite skew brace} is
$A_{\mathrm{op}}=(A,\cdot_{\mathrm{op}},\bul)$, where
$a\cdot_{\mathrm{op}}b=b\cdot a$. Its $\star$-product is $a\star_{\mathrm{op}}b=b^{-1}\cdot(a\bul b)\cdot a^{-1}$. The ideals of $A$ and $A_{\mathrm{op}}$ coincide. In particular,
$A\star A$ and $A\star_{\mathrm{op}}A$ are ideals of $A$;
the corresponding quotients are the \textit{trivial} $(A/A\star A,\cdot,\cdot)$ and \textit{almost trivial} $(A/A\star_{\mathrm{op}}A,\cdot,\cdot^{\mathrm{op}})$ skew braces,
respectively, see \cite[Section~2]{Trappeniers}.

The following identities for the $\star$-product will be used throughout. They are collected, for example, in \cite[Lemma~2.7]{Tsang}.

\begin{lemma}\label{lem:fund}
Given a skew brace $(A,\ld,\bul)$, the following equalities hold for all $a,b,c \in A$:
\begin{enumerate}
   \item\label{L2} $a \star (b \ld c) = (a \star b) \ld b \ld (a \star c) \ld b^{-1}$,
  \item\label{L3} $(a \bul b) \star c
        = \bigl(a \star (b \star c)\bigr) \ld (b \star c) \ld (a \star c)$,
    \item $\lambda_{a}(b\star c)=(a\bullet b\bullet \overline{a})\star\lambda_{a}(c)$,
    \item $a\bullet b\bullet \overline{a}=a\ld \lam_{a}(b\cdot(b\star \overline{a}))\cdot a^{-1}$.
\end{enumerate}
\end{lemma}

For a brace, it is known that associativity of $\star$ is equivalent to two-sidedness.

\begin{theorem}[\cite{Rump,Lau}]\label{thm:rump-lau}
Given a brace $(A,\ld,\bul)$, the following are equivalent:
\begin{enumerate}[label=\textup{(\roman*)}]
  \item[1)]\label{rl:assoc} $\star$ is associative.
  \item[2)]\label{rl:twosided} $A$ is two-sided.
\end{enumerate}
\end{theorem}

In the next section, we extend this result beyond the brace setting.

\section{Two-sidedness and associativity of the $\star$-product}

First, we observe the following straightforward result.

\begin{lemma}\label{lem:conditiontwosided}
A skew brace $(A,\cdot,\bullet)$ is two-sided if and only if 
\begin{equation}\label{eqcondtwosided}
(a\cdot b)\star c=b^{-1}\cdot(a\star c)\cdot b\cdot(b\star c), \qquad \text{for all}\ a,b,c\in A.
\end{equation}
\end{lemma}
\begin{proof}
Condition \eqref{eq:twosided} reads $a \cdot b\cdot\lambda_{a\cdot b}(c)=a\cdot\lambda_{a}(c)\cdot c^{-1}\cdot b\cdot\lambda_{b}(c)$, which is equivalent to $b\cdot((a\cdot b)\star c)\cdot c=(a\star c)\cdot b\cdot(b\star c)\cdot c$, hence to \eqref{eqcondtwosided}.
\end{proof}

\begin{remark}\label{rmk:twosidedditr}
If $A\star A\subseteq Z(A^{\cdot})$, then $A$ is two-sided if and only if $\star$ is right-distributive with respect to $\cdot$. This is the case for braces, where the condition  $A\star A\subseteq Z(A^{\cdot})$ is trivially satisfied.
\end{remark}

We now prove that, under the assumption $A\star A\subseteq Z(A^{\cdot})$, Theorem~\ref{thm:rump-lau} can be lifted. First, we prove the easier direction.

\begin{proposition}\label{thm:starassocitiveskewbrace}
Let $(A,\cdot,\bullet)$ be a skew brace such that $A\star A\subseteq Z(A^{\cdot})$. If $A$ is two-sided then $\star$ is associative.
\end{proposition}

\begin{proof}
By Remark \ref{rmk:twosidedditr}, since $A\star A\subseteq Z(A^{\cdot})$, two-sidedness is equivalent to $(x\cdot y)\star c=(x\star c)\cdot(y\star c)$. Since $a\bullet b=a\cdot(a\star b)\cdot b$, by applying right-distributivity twice, we get
\[
(a\bullet b)\star c=(a\cdot(a\star b)\cdot b)\star c=(a\star c)\cdot((a\star b)\star c)\cdot(b\star c)
\]
while by 2) of Lemma \ref{lem:fund}, we have $(a \bul b) \star c= \bigl(a \star (b \star c)\bigr) \ld (b \star c) \ld (a \star c)$. Using again $A\star A\subseteq Z(A^{\cdot})$, we get that $\star$ is associative.
\end{proof}

We also have the other direction, that generalizes the result achieved by Lau in the brace setting \cite{Lau}.

\begin{theorem}\label{thm:seconddirection}
Let \((A,\cdot,\bullet)\) be a skew brace such that \(A\star A\subseteq Z(A^\cdot)\). If \(\star\) is associative, then 
\begin{equation}\label{eq:rigtdistributivity}
(x\cdot y)\star c=(x\star c)\cdot(y\star c),\qquad \text{for all}\ x,y,c\in A.
\end{equation}
In particular, by Remark \ref{rmk:twosidedditr}, $A$ is two-sided.
\end{theorem}

\begin{proof}
We have $a \bul a^{-1} \;=\; a \ld (a \star a^{-1}) \ld a^{-1}=a \star a^{-1}$. By 2) of Lemma \ref{lem:fund}, we get 
\[
\begin{split}
    (a\star a^{-1})\star c&=(a\bul a^{-1})\star c=\bigl(a \star (a^{-1} \star c)\bigr) \ld (a^{-1}\star c) \ld (a \star c)\\&=\bigl((a \star a^{-1}) \star c\bigr) \ld (a^{-1}\star c) \ld (a \star c)
\end{split}
\]
and then $a^{-1}\star c=(a\star c)^{-1}$. Given $a,b\in A$, we observe that $\lambda_{\overline{a}}(b)\cdot b^{-1}=\overline{a}\star b\in Z(A^{\cdot})$, hence $\lambda_{\overline{a}}(b)$ and $b$ commute. Therefore, we get
\begin{align}\label{auxeq}
(\overline{a}\bullet b^{-1})\cdot b&=\overline{a}\cdot\lambda_{\overline{a}}(b)^{-1}\cdot b=\overline{a}\cdot b\cdot \lambda_{\overline{a}}(b^{-1})=\overline{a}\cdot\lambda_{\overline{a}}(\lambda_{a}(b)\cdot b^{-1})\\&=\overline{a}\bul(a\star b)\nonumber.
\end{align}
We also have
\begin{equation}\label{auxeq2}
\lambda_{a\star b}(c)
 =((a\star b)\star c)\cdot c=(a\star(b\star c))\cdot c=\lambda_a(b\star c)\cdot(b\star c)^{-1}\cdot c
\end{equation}
so that 
\[
\lambda_{(\overline{a}\bullet b^{-1})\cdot b}(c)\overset{\eqref{auxeq}}{=}\lambda_{\overline{a}\bul(a\star b)}(c)=\lambda_{\overline{a}}\lambda_{a\star b}(c)\overset{\eqref{auxeq2}}{=}(b\star c)\cdot\lambda_{\overline{a}}(b\star c)^{-1}\cdot\lambda_{\overline{a}}(c).
\]
Moreover, using $a^{-1}\star c=(a\star c)^{-1}$, we obtain 
\[
\lambda_{\overline{a}\bul b^{-1}}(c)=\lambda_{\overline{a}}\lambda_{b^{-1}}(c)=\lambda_{\overline{a}}((b^{-1}\star c)\cdot c)=\lambda_{\overline{a}}(b\star c)^{-1}\cdot \lambda_{\overline{a}}(c),
\]
hence
\[
\lambda_{(\overline{a}\bullet b^{-1})\cdot b}(c)=(b\star c)\cdot\lambda_{\overline{a}\bul b^{-1}}(c).
\]
As a consequence, $\cdot$-multiplying on the right by $c^{-1}$, we get
\[
((\overline{a}\bullet b^{-1})\cdot b)\star c=(b\star c)\cdot((\overline{a}\bul b^{-1})\star c)=((\overline{a}\bul b^{-1})\star c)\cdot(b\star c)
\]
and, since for a fixed $b$ the term $\overline{a}\bul b^{-1}$ covers all $A$, we get \eqref{eq:rigtdistributivity}.
\end{proof}

\begin{remark}
In general, it is known that the associativity of $\star$ does not imply two-sidedness, see \cite[\S 2.2]{KSV}. In Proposition \ref{prop:index-two} we provide a constructive method to realize skew braces which are not two-sided and whose $\star$ operation is associative.
\end{remark}

We can obtain something more than Proposition \ref{thm:starassocitiveskewbrace} and Theorem \ref{thm:seconddirection}. To do this, we first study the case of almost trivial skew braces $(G,\cdot,\cdot^{\mathrm{op}})$. In this setting, the $\star$-product becomes
\[
  a \star b
  = a^{-1}\cdot(a \bul b)\cdot b^{-1}
  = a^{-1}\cdot b\cdot a\cdot b^{-1}
  = [a,\,b^{-1}],
\]
where $[x,y] = x^{-1}y^{-1}xy$ is the classical group commutator. The skew brace $(G,\ld,\ld^{\mathrm{op}})$ is always two-sided:
\[
(a\cdot b)\bullet c=c\cdot a\cdot b=c\cdot a\cdot c^{-1}\cdot c\cdot b=(a \bul c)\cdot c^{-1} \ld (b \bul c).
\]
We consider the lower series for groups: $\gamma_{1}(G):=G$ and $\gamma_{n+1}(G):=[\gamma_{n}(G),G]$ for $n\geq1$. For an almost trivial skew brace, the left and right \(\star\)-series coincide with the lower central series of the underlying group, see e.g.\ \cite[Section~3]{Tsang}. In particular, if \(\gamma_3(G)=1\), then all triple \(\star\)-products vanish. Proposition \ref{prop:quasi} shows that, conversely, associativity of the $\star$-product alone already forces $\gamma_3(G)=1$. First we recall some identities that will be useful.

\begin{remark}\label{rmk:commutatoridentities}
Given a group $G$, the following equalities are satisfied:
\[
[x,yz]=[x,z][x,y]^{z},\qquad [xy,z]=[x,z]^{y}[y,z],
\]
where $g^{h}=h^{-1}gh$, for all $h,g\in G$. From these equalities, one obtains $1=[x,yy^{-1}]=[x,y^{-1}][x,y]^{y^{-1}}$, from which one gets
\[
[x,y^{-1}]=([x,y]^{y^{-1}})^{-1}=([x,y]^{-1})^{y^{-1}}=[x,y]^{-1}[[x,y]^{-1},y^{-1}],
\]
i.e. $[x,y^{-1}]=[x,y]^{-1}z$, with $z\in\gamma_{3}(G)$.
\end{remark}

The following result may be viewed as an analogue, for the operation $a\star b=[a,b^{-1}]$, of the classical results on associativity and Engel identities for group commutators.

\begin{proposition}\label{prop:quasi}
Let $(G,\ld)$ be an arbitrary group and consider the almost trivial skew brace $(G,\ld,\ld^{\mathrm{op}})$. Then $\star$ is associative if and only if $\gamma_3(G) = 1$, i.e.\ $G$ has
nilpotency class at most $2$.
\end{proposition}

\begin{proof}
We recall that $a \star b = [a,b^{-1}]$, for all $a,b\in G$.

\smallskip\noindent
\textbf{($\Leftarrow$)} Suppose $\gamma_3(G)=1$. 
Then, for any $u \in \gamma_2(G)$ and $v\in G$, we have $[u,v]=1$.  Hence $a \star (b \star c) = [a,[b,c^{-1}]^{-1}] = 1$
and $(a \star b) \star c = [[a,b^{-1}],c^{-1}]=1$, so all triple products vanish
and $\star$ is trivially associative.

\smallskip\noindent
\textbf{($\Rightarrow$)} Assume that $\star$ is associative. Taking $b=c$ and using $b\star b=[b,b^{-1}]=1$ together with $a\star1=1$ we get
\[
[[a,b^{-1}],b^{-1}]=(a\star b)\star b=a\star(b\star b)=a\star 1=1,
\]
for all $a,b\in G$. Therefore, the group is 2-Engel and every 2-Engel group is nilpotent of class at most 3 \cite{Levi,Hopkins}, so $\gamma_{4}(G)=1$. 

Since $\gamma_3(G)\subseteq Z(G)$, for $u\in\gamma_2(G)$ and $v\in G$
we have
\[
[u^{-1},v]=[u,v]^{-1}=[u,v^{-1}].
\]
Moreover, the commutator identities recalled in Remark \ref{rmk:commutatoridentities} give
\[
[a,b^{-1}]\equiv[a,b]^{-1}\pmod{\gamma_3(G)},\qquad
[b,c^{-1}]^{-1}\equiv[b,c]\pmod{\gamma_3(G)}.
\]
As the omitted factors are central, they disappear in further
commutators. Consequently,
\begin{align*}
(a\star b)\star c
 &= [[a,b^{-1}],c^{-1}]
  = [[a,b]^{-1},c^{-1}]=[[a,b],c],\\
a\star(b\star c)
 &= [a,[b,c^{-1}]^{-1}]=[a,[b,c]].
\end{align*}
The Hall--Witt identity gives
\[
\begin{split}
1&=[[a,b^{-1}],c]^{b}[[b,c^{-1}],a]^{c}[[c,a^{-1}],b]^{a}=[[a,b^{-1}],c][[b,c^{-1}],a][[c,a^{-1}],b]\\&=[[a,b],c]^{-1}[[b,c],a]^{-1}[[c,a],b]^{-1},
\end{split}
\]
since the triple commutators are central.
Hence we have
\begin{equation}\label{finalequation}
[a,[b,c]]=[[b,c],a]^{-1}=[[a,b],c]\cdot[[c,a],b].
\end{equation}
Associativity and the preceding computations give
$[[a,b],c]=[a,[b,c]]$. Thus, from \eqref{finalequation}, we get
$[[c,a],b]=1$ for all $a,b,c\in G$, so $\gamma_3(G)=1$.
\end{proof}

\begin{corollary}\label{cor:almosttrivialstarassoc}
Let $G$ be a group with $\gamma_{3}(G)\not=1$.
Then the almost trivial skew brace on $G$ is two-sided but $\star$ is not associative.
\end{corollary}

\begin{remark}\label{rmk:almosttriavilstarass}
The smallest example is the almost trivial skew brace over $S_{3}$, of order 6: every group of order less than 6 is abelian, and for a brace two-sidedness is equivalent to $\star$ being associative by Theorem \ref{thm:rump-lau}. In this case $S_{3}\star S_{3}=[S_3,S_3]=A_3$.
\end{remark}

Proposition~\ref{prop:quasi} allows us to prove the following extension of Theorem~\ref{thm:rump-lau} to skew braces.

\begin{theorem}\label{thm:extension}
Let $(A,\cdot,\bullet)$ be a skew brace. Any two of the following conditions imply the third:
\begin{itemize}
    \item[1)] $A$ is two-sided;
    \item[2)] $\star$ is associative;
    \item[3)] $A\star A\subseteq Z(A^\cdot)$.
\end{itemize}
\end{theorem}

\begin{proof}
By Proposition \ref{thm:starassocitiveskewbrace}, we already know that 1) and 3) imply 2), while by Theorem \ref{thm:seconddirection} we know that 2) and 3) imply 1). Therefore, it remains to be proven that 1) and 2) imply 3).

Set $B:=A\star A$ and $J:=A\star_{\mathrm{op}}A$. Both are ideals of $A$,
and two-sidedness gives $[B,J]_{\cdot}=1$ by
\cite[Lemma~4.1]{Trappeniers}. 
Moreover, $A/J$ is almost trivial and the quotient map preserves $\star$, hence its induced $\star$-product is associative. By
Proposition~\ref{prop:quasi}, $\gamma_{3}(A/J)=1$. In the almost trivial skew brace $A/J$, the maps $\lambda_{aJ}$ are inner automorphisms of $(A/J)^\cdot$, while the $\star$-products belong to its commutator subgroup, which is central since $\gamma_3(A/J)=1$. Hence $\lambda_{aJ}$ fixes each $(b\star c)J$, and therefore fixes the subgroup $BJ/J$ they generate.
It follows that $a\star z\in J$ for every $z\in B$.
Since $a\star z\in B$ as well, we obtain
\[
a\star B\subseteq I:=B\cap J\subseteq Z(B).
\]
For $u,v\in B$, by (1) of Lemma~\ref{lem:fund}, we get
\begin{equation}\label{eqqauxiliary}
a\star(u\cdot v)=(a\star u)\cdot u\cdot (a\star v)\cdot u^{-1}=(a\star u)\cdot(a\star v).
\end{equation}
Thus $a\star(-)|_B:B\to Z(B)$ is a morphism of groups.

Therefore, using again (1) of Lemma~\ref{lem:fund}, we have
\[
\begin{split}
a\star(b\star(x\cdot y))&=a\star\big((b \star x) \ld x \ld (b \star y) \ld x^{-1}\big)\\&\overset{\eqref{eqqauxiliary}}{=}(a\star(b\star x))\cdot(a\star(x \ld (b \star y) \ld x^{-1}))
\end{split}
\]
but, since $\star$ is associative, we also have 
\[
\begin{split}
a\star(b\star(x\cdot y))&=(a\star b)\star(x\cdot y)=((a\star b) \star x) \ld x \ld ((a\star b) \star y) \ld x^{-1}\\&=(a\star(b\star x)) \ld x \ld (a\star (b \star y)) \ld x^{-1}
\end{split}
\]
and then
\begin{equation}\label{eq:centraleq}
x \ld (a\star (b \star y)) \ld x^{-1}=a\star(x \ld (b \star y) \ld x^{-1}).
\end{equation}
For fixed $a,x\in A$, the maps
$z\mapsto x\cdot(a\star z)\cdot x^{-1}$ and
$z\mapsto a\star(x\cdot z\cdot x^{-1})$ are morphisms of groups on $B$:
the latter is defined on $B$ because $B$ is normal in $A^{\cdot}$.
The equality \eqref{eq:centraleq} shows that they agree on every generator
$b\star y$ of $B$, so they agree on all of $B$.
This implies
\begin{equation}
x\cdot(a\star z)\cdot x^{-1}=a\star(x\cdot z\cdot x^{-1}),\qquad \text{for all}\ x\in A,z\in B.
\end{equation}
By multiplying by $x\cdot z\cdot x^{-1}$ on the right, we get
\[
x\cdot\lambda_{a}(z)\cdot x^{-1}=\lambda_{a}(x\cdot z\cdot x^{-1})=\lambda_{a}(x)\cdot\lambda_{a}(z)\cdot\lambda_{a}(x)^{-1}.
\]
This means that $a\star x$ commutes with $x\cdot\lambda_{a}(z)\cdot x^{-1}$. Since $B$ is normal and $\lambda_{a}(B)=B$, the elements $x\cdot\lambda_{a}(z)\cdot x^{-1}$ with $z\in B$ cover all $B$. Hence $a\star x\in Z(B)$ and they are generators, then $B$ is abelian. Therefore, we get
\[
\begin{split}
    (a\bullet b)\star c&=(a\cdot(a\star b)\cdot b)\star c\\&\overset{\eqref{eqcondtwosided}}{=}b^{-1}\cdot(a\star b)^{-1}\cdot(a\star c)\cdot(a\star b)\cdot((a\star b)\star c)\cdot b\cdot(b\star c)\\&=b^{-1}\cdot(a\star c)\cdot((a\star b)\star c)\cdot b\cdot(b\star c).
\end{split}
\]
On the other hand, by 2) of Lemma \ref{lem:fund} and associativity of $\star$ we get
\[
(a \bul b) \star c
        = \bigl(a \star (b \star c)\bigr) \ld (b \star c) \ld (a \star c)=(a \star c)\ld\bigl((a \star b) \star c\bigr)\ld(b \star c),
\]
so that 
\[
b^{-1}\cdot(a\star c)\cdot((a\star b)\star c)\cdot b=(a \star c)\ld\bigl((a \star b) \star c\bigr).
\]
But we also have that 
\[
\begin{split}
a\star\lambda_{b}(c)&=a\star((b\star c)\cdot c)=(a \star (b\star c)) \ld (b\star c) \ld (a \star c) \ld (b\star c)^{-1}\\&=(a \star (b\star c)) \ld (a \star c)=(a \star c)\ld ((a \star b)\star c) 
\end{split}
\]
and then $b\cdot(a\star\lambda_{b}(c))=(a\star\lambda_{b}(c))\cdot b$. Since $\lambda_{b}$ is surjective, we have that $b$ commutes with any $a\star x$ and then, since $b\in A$ is arbitrary, $B=A\star A\subseteq Z(A^{\cdot})$.
\end{proof}

\begin{corollary}\label{app:generators}
Let $(A,\cdot,\bullet)$ be a two-sided skew brace such that $A^{\cdot}$ is generated by
 $s_1,\ldots,s_d$. Then $\star$ is associative if and only if
\begin{equation}\label{condcommutators}
    [s_i\star s_j,s_k]_{\cdot}=1
 \qquad 1\leq i,j,k\leq d.
\end{equation}
 When these conditions hold,
 $A\star A=\langle s_i\star s_j\mid1\leq i,j\leq d\rangle_{\cdot}$.
 \end{corollary}
\begin{proof}
Let $N$ be the normal closure in $A^{\cdot}$ of the $s_i\star s_j$. By (1) of Lemma \ref{lem:fund} and \eqref{eqcondtwosided}, together
 with their consequences
 \[
 a\star b^{-1}=b^{-1}\cdot(a\star b)^{-1}\cdot b,\qquad
 a^{-1}\star b=a\cdot(a\star b)^{-1}\cdot a^{-1},
 \]
 allow us to expand words in both arguments. First expanding the second
 argument with the first fixed at $s_i$, and then expanding the first
 argument, shows that every $a\star b$ belongs to $N$. Since $A\star A$
 is normal and contains the generators of $N$, we have $N=A\star A$.
 The commutator conditions \eqref{condcommutators} say that each $s_i\star s_j$
 commutes with all generators of $A^{\cdot}$, and hence is central.
 They therefore imply $N\subseteq Z(A^{\cdot})$, and
Theorem~\ref{thm:extension} gives associativity. The converse follows
 from the same theorem. Under these conditions the normal closure is
 already the subgroup generated by the $s_i\star s_j$.
 \end{proof}

We now provide a constructive method to obtain $\star$-associative skew braces that are not two-sided.

\begin{proposition}\label{prop:index-two}
Let $(G,\cdot)$ be a group, $H\leq G$ be a subgroup of index $2$, and $\alpha\in\Aut(G)$ satisfying $\alpha^2=\mathrm{Id}$ and $\alpha(H)=H$.

Define $\lambda_a=\mathrm{Id}$ for $a\in H$ and $\lambda_a=\alpha$
for $a\notin H$. Then $a\bul b:=a\cdot\lambda_a(b)$ defines a skew
brace on $G$, with
\begin{equation}\label{defstarproduct}
a\star b=
\begin{cases}
1,&a\in H,\\
\alpha(b)\cdot b^{-1},&a\notin H.
\end{cases}
\end{equation}
Put $\Delta:=\{\alpha(c)\cdot c^{-1}\mid c\in G\}$. Then
$\Delta\subseteq H$, and the following hold:
\begin{enumerate}[label=\textup{\arabic*)}]
\item $\star$ is associative if and only if
$\Delta\subseteq\Fix(\alpha)$ if and only if $d^2=1$
for every $d\in\Delta$. In this case all triple $\star$-products are trivial.
\item The skew brace is two-sided if and only if
\[
b^{-1}\cdot d\cdot b=d\quad(d\in\Delta,\ b\in H),\qquad
b^{-1}\cdot d\cdot b=d^{-1}\quad(d\in\Delta,\ b\notin H).
\]
In particular, when $\star$ is associative, two-sidedness is equivalent
to $\Delta\subseteq Z(G)$.
\end{enumerate}
\end{proposition}
\begin{proof}
The map $\lambda$ is the composite of the quotient of groups $G\to G/H\cong C_2$ and the morphism of groups $C_2\to\Aut(G)$ sending
its generator to $\alpha$. This description also covers
$\alpha=\mathrm{Id}$. Since $\alpha$ preserves the two cosets of $H$,
$\lambda_{\lambda_a(b)}=\lambda_b$. As a consequence, $\lambda_{a\bul b}
=\lambda_{a\cdot\lambda_a(b)}=\lambda_a\lambda_b$ and
\[
(a\bul b)\bul c
=a\cdot\lambda_a(b)\cdot\lambda_a\lambda_b(c)
=a\cdot\lambda_a(b\bul c)=a\bul(b\bul c).
\]
The identity is $1$. The element
$\overline a=\lambda_a^{-1}(a^{-1})$ is a right inverse for $a$, and
$\lambda_{\overline a}=\lambda_a^{-1}$ shows that it is also a left
inverse. The compatibility condition \eqref{eq:brace-axiom} follows because each $\lambda_a$ is
an automorphism of $(G,\cdot)$, and \eqref{defstarproduct} is immediate.

\noindent 1). Write $\delta(c)=\alpha(c)\cdot c^{-1}$. The induced automorphism of
$G/H\cong C_2$ is the identity, so $\delta(c)\in H$ for every $c$. Moreover $\alpha(\delta(c))=c\cdot\alpha(c)^{-1}=\delta(c)^{-1}$. Therefore, $(a\star b)\star c=1$ for all $a,b,c\in G$, while
\[
a\star(b\star c)=
\begin{cases}
\delta(c)^{-2},&a,b\notin H,\\
1,&a\in H\text{ or }b\in H.
\end{cases}
\]

\noindent 2). We apply Lemma~\ref{lem:conditiontwosided}.
Its identity holds automatically when $a\in H$. When $a\notin H$
and $b\in H$, it becomes $d=b^{-1}\cdot d\cdot b$, with $d=\delta(c)$.
When $a,b\notin H$, it becomes $1=b^{-1}\cdot d\cdot b\cdot d$. These are precisely
the stated conditions. Under the associativity assumption, $d=d^{-1}$, so together they require $\Delta\subseteq Z(G)$.
\end{proof}

\begin{example}\label{ex:assnottwosided}
Let $G=\langle r,s\ |\ r^{8}=s^{2}=1,\ srs=r^{5}\rangle$ be the modular group of order 16. Let $H=\langle r^{2},s\rangle$ and $\alpha\in\mathrm{Aut}(G)$ be given by $\alpha(r)=rs$, $\alpha(s)=s$. Then $\alpha^{2}=\mathrm{id}$, $\alpha(H)=H$ and
\[
\Delta=\{1,s,r^{4},r^{4}s\}\subseteq H\cap\mathrm{Fix}(\alpha),
\]
so all triples are trivial and $\star$ is associative. It is \textit{not} two-sided: for $a=b=c=r\notin H$ one has $a\star c=b\star c=rsr^{-1}=r^{4}s$ and
\[
(a\cdot b)\star c=r^{2}\star r=1,\quad  b^{-1}\cdot(a\star c)\cdot b\cdot(b\star c)=r^{7} r^{4}srr^{4}s=r^{4}.
\]
We observe that $Z(G)=\{1,r^2,r^4,r^6\}$.
\end{example}


\section{Applications for two-sided skew braces}\label{sec:applications}

Unless otherwise stated, throughout this section $(A,\cdot,\bullet)$ is a two-sided skew brace with associative $\star$-product. By Theorem \ref{thm:extension}, we have that $A\star A\subseteq Z(A^{\cdot})$. Therefore, the following equalities hold:
\begin{align}
    a\star(b\cdot c)&=(a\star b)\cdot(a\star c),\label{eq:distrleft}\\
    (a\cdot b)\star c&=(a\star c)\cdot(b\star c).\label{eq:distrright}
\end{align}
Thus, for any $a,c\in A$, we have morphisms of groups $a\star(-),(-)\star c:A^{\cdot}\to Z(A^{\cdot})$. Hence, defining $D:=[A^{\cdot},A^{\cdot}]$, we get $A\star D=\{1\}=D\star A$ or, equivalently, $D\subseteq\mathrm{ker}(\lambda)\cap\mathrm{Fix}(\lambda)$. Clearly $D$ is an ideal of the skew brace $(A,\cdot,\bullet)$. In fact, it is normal in $A^{\cdot}$ and $\lambda_{a}(d)=d$, for all $a\in A$ and $d\in D$. Therefore, for all $a\in A$ and $d\in D$, we get $a\bul d=a\cdot\lambda_{a}(d)=a\cdot d$ and $d\bul a=d\cdot\lambda_{d}(a)=d\cdot a$. The two operations then coincide on $D$. We get 
\[
a\bul d\bullet\overline{a}=(a\cdot d)\bul\overline{a}=(a\cdot d\cdot a^{-1}\cdot a)\bul\overline{a}=(a\cdot d\cdot a^{-1})\bul a\bul\overline{a}=a\cdot d\cdot a^{-1}\in D,
\]
hence $D$ is normal in $A^{\bul}$. 

Recalling that a skew brace $(A,\cdot,\bullet)$ is \textit{semiprime} if, for every ideal $K$, the condition $K\star K=\{1\}$ implies $K=\{1\}$, we get the following result.

\begin{corollary}
Let $A$ be a semiprime two-sided skew brace. Then $\star$ is associative if and only if $A^{\cdot}$ is abelian.
\end{corollary}

\begin{proof}
If $\star$ is associative then we have that $D=[A^{\cdot},A^{\cdot}]$ is an ideal of $A$ satisfying $A\star D=\{1\}=D\star A$, hence $D\star D=\{1\}$. Since $A$ is semiprime, we get $D=\{1\}$, hence $A^{\cdot}$ is abelian. 

The other direction does not need the fact that $A$ is semiprime and is exactly Theorem \ref{thm:rump-lau}.
\end{proof}

On the quotient $R:=A/D$ we can define
\[
(aD)+(bD):=(a\cdot b)D,\qquad (aD)\cdot(bD):=(a\star b)D.
\]
The operation $+$ is abelian, the multiplication of $R$ is associative because it is induced by $\star$, and the two distributivities follow by \eqref{eq:distrleft} and \eqref{eq:distrright}. Finally, the Jacobson operation $(aD)\bullet(bD)=aD+(aD)\cdot(bD)+bD=(a\cdot(a\star b)\cdot b)D=(a\bullet b)D$, so we recover precisely $A^{\bullet}/D$. We have that $(R,+,\cdot)$ is a radical ring. We write $\pi:A\to R$ for the quotient map, viewed as a morphism of groups from $A^{\cdot}$ to $R^{+}$.

\begin{proposition}\label{app:bilinear}
The map
\[
\beta:R^{+}\times R^{+}\longrightarrow Z(A^{\cdot}),
\ (aD,bD)\longmapsto a\star b,
\]
is well-defined and biadditive. Moreover, for all $r,s,t\in R$,
\[
\pi(\beta(r,s))=rs,\qquad
\beta(rs,t)=\beta(r,st).
\]
There is a canonical
surjective morphism of abelian groups
\[
\widetilde\beta:
R^{+}\otimes_{\mathbb Z}R^{+}
\longrightarrow A\star A,\ (aD)\otimes(bD)\longmapsto a\star b.
\]
\end{proposition}
\begin{proof}
For $a,b\in A$ and $d,d'\in D$, equations \eqref{eq:distrleft} and \eqref{eq:distrright} and
$A\star D=D\star A=\{1\}$ give
\[
(a\cdot d)\star(b\cdot d')
=(a\star b)\cdot(a\star d')\cdot(d\star b)\cdot(d\star d')
=a\star b.
\]
This proves well-definedness and biadditivity follows from the same identities.
The first displayed equality is the definition of multiplication in $R$. The second one follows by associativity of $\star$.

We apply the universal property of the tensor product to the biadditive
map $\beta$ getting the displayed morphism of abelian groups $\tilde{\beta}$.

Its values generate $A\star A$,
which proves surjectivity. 
\end{proof}

\begin{corollary}\label{app:tensor}
The following statements hold:
\begin{enumerate}[label=\textup{\arabic*)}]
\item If
$\operatorname{Hom}_{\mathbb Z}
(R^{+}\otimes_{\mathbb Z}R^{+},Z(A^{\cdot}))=0$,
then $A$ is a trivial skew brace. In particular, this holds when $R^{+}$
is torsion and $Z(A^{\cdot})$ is torsion-free.
\item If $R^{+}$ is finite, then $A\star A$ is finite and
its exponent divides the exponent of $R^{+}$.
\end{enumerate}
\end{corollary}

\begin{remark}\label{app:central-kernel}
Although $D$ is annihilated on both sides by $\star$, it need not be
contained in $Z(A^{\cdot})$: a trivial skew brace on a group whose derived
subgroup is not central provides an example.
In contrast, put $B:=A\star A$ and $I:=B\cap D$. Then
\[
I\subseteq Z(A^{\cdot})\cap\Ker(\lambda)\cap\Fix(\lambda)=\mathrm{Ann}(A).
\]
The operations $x+y:=x\cdot y$ and $xy:=x\star y$ make $B$ a radical
ring, and restriction of $\pi$ gives an exact sequence of rings
\[
0\longrightarrow I\longrightarrow B\longrightarrow R^2
\longrightarrow0,
\qquad BI=IB=0.
\]
Indeed, $B$ is central in $A^{\cdot}$ and is closed under $\star$.
It is also a subbrace: for $b\in B$,
$\lambda_b(a)=(b\star a)\cdot a$ induces the identity on
$A^{\cdot}/B$, so both $\lambda_b$ and its inverse preserve $B$;
hence $\overline b=\lambda_b^{-1}(b^{-1})\in B$.
The restricted operations therefore give a radical ring.
The image of $B$ under $\pi$ is $R^2$, the kernel is $I$, and the
annihilation assertion follows from $I\subseteq D$.
\end{remark}

With $B$, $D$ and $I$ defined as above, we set
\[
T:=A/B,\qquad R:=A/D,\qquad C:=A/(BD).
\]
We know that $T$ is a trivial skew brace and here $R$ is the skew brace associated
with the radical ring defined above. We are now able to describe $A/I$ through $T$, $R$ and $C$.

\begin{proposition}\label{app:fibre-product}
We have that $C\cong R/R^2$ is a trivial
brace. The natural maps to $C$ give a
canonical isomorphism of skew braces
\[
A/I\ \cong\ T\times_C R.
\]
In particular, $A$ is a central extension of this fibre product, with
kernel $I\subseteq\operatorname{Ann}(A)$.
\end{proposition}
\begin{proof}
The product $BD$ is an ideal of $A$ since $B$ and $D$ are ideals, and its
image in $R$ is $R^2$. This proves $C\cong R/R^2$. Let $q_T:T\to C$ and $q_R:R\to C$ be the natural quotient maps. The
fibre product is the subbrace
\[
T\times_C R=\{(t,r)\in T\times R\mid q_T(t)=q_R(r)\}.
\]
The morphism $A\to T\times_{C}R:a\mapsto(aB,aD)$ has kernel $I$. To see that it is onto, take $(xB,yD)$ with
$xBD=yBD$. Write $x^{-1}y=bd$ with $b\in B$, $d\in D$.
Then $a=xb$ satisfies $aB=xB$ and $aD=yD$.
The first isomorphism theorem now gives the asserted isomorphism of skew braces.
The centrality of the kernel follows from
Remark~\ref{app:central-kernel}.
\end{proof}

Now we define
\[
P_1(A):=A,\qquad
P_n(A):=\langle a_1\star\cdots\star a_n\mid a_i\in A\rangle_{\cdot}
\quad n\geq2.
\]
Associativity of $\star$ makes the placement of parentheses irrelevant. With $n\geq2$, we have
\[
\{1\}\subseteq\cdots\subseteq P_{n+1}(A)\subseteq\cdots\subseteq P_{2}(A)=A\star A\subseteq Z(A^{\cdot})\subseteq P_{1}(A)=A.
\]
\begin{remark}\label{rmk:identityPistar}
We observe that $P_{i}(A)\star P_{j}(A)=P_{i+j}(A)$. Given an arbitrary element $u=\prod_{\alpha}{u_{\alpha}^{\epsilon_{\alpha}}}\in P_{i}(A)$, where $\epsilon_{\alpha}\in\{\pm1\}$ and $u_{\alpha}$ are $\star$-products of length $i$, and an arbitrary element $v=\prod_{\beta}{v_{\beta}^{\eta_{\beta}}}\in P_{j}(A)$, where $\eta_{\beta}\in\{\pm1\}$ and $v_{\beta}$ are $\star$-products of length $j$, by \eqref{eq:distrleft} and \eqref{eq:distrright} we get $u\star v=\prod_{\alpha,\beta}(u_{\alpha}\star v_{\beta})^{\epsilon_{\alpha}\eta_{\beta}}$ which is in $P_{i+j}(A)$ as any factor belongs to it by associativity of $\star$. On the other hand, any generator of $P_{i+j}(A)$ can be written as $(a_{1}\star\cdots\star a_{i})\star(a_{i+1}\star\cdots\star a_{i+j})$, which belongs to $P_{i}(A)\star P_{j}(A)$. 
\end{remark}

We can now compare $P_{n}(A)$ with left, right and strong $\star$-series introduced in \cite{Rump} for braces and later extended to all skew braces in \cite{CSV}. More explicitly, set $L_1=Q_1=S_1=A$ and for $n\geq1$
\begin{gather*}
L_{n+1}=A\star L_n,\qquad Q_{n+1}=Q_n\star A,\\
S_{n+1}=\left\langle S_i\star S_{n+1-i}\mid1\leq i\leq n
\right\rangle_{\cdot}.
\end{gather*}

The identity $P_i(A)\star P_j(A)=P_{i+j}(A)$ proved in Remark \ref{rmk:identityPistar} gives the
three equalities of series by induction. In particular, left, right and strong $\star$-nilpotence of $A$ are
equivalent.

We obtain the following result:

\begin{proposition}\label{prop:powers}
For every $n\geq 1$, there is a short
exact sequence of groups
\begin{equation}\label{exactsequence}
1\longrightarrow P_{n}(A)\cap D\longrightarrow P_{n}(A)\overset{\pi}\longrightarrow R^{n}\longrightarrow1,
\end{equation}
where $R^n$ is regarded as an additive group. 

Moreover, $P_n(A)$ is an ideal of $A$ and, for $n\geq 2$, the operations
\[
x+y:=x\cdot y,\qquad xy:=x\star y
\]
make $P_n(A)$ into a radical ring, and \eqref{exactsequence} is indeed a short exact sequence of rings, where $R^n$ carries the ring structure induced from $R$. The kernel
$P_n(A)\cap D$ is annihilated on both sides by $A$ under $\star$; in
particular, it has zero multiplication.
\end{proposition}

\begin{proof}
The projection $\pi\colon A\to R$ preserves the $\star$-product.
Hence it sends the generators of $P_n(A)$ onto the additive generators
of $R^n$, and therefore $\pi(P_n(A))=R^n$.
Its kernel on $P_n(A)$ is $P_n(A)\cap D$, which gives the short exact sequence of groups \eqref{exactsequence}. 


Moreover, each $P_n(A)$ is an ideal of $(A,\cdot,\bullet)$, since the terms of the right $\star$-series are ideals by \cite[Proposition 2.1]{CSV}.
For $n\geq2$, its additive group is abelian. Furthermore, $P_n(A)\star P_n(A)=P_{2n}(A)\subseteq P_n(A)$, and the two distributive laws and associativity of $\star$ are
inherited from $A$. Since $P_n(A)$ is a subbrace, these operations
make it a radical ring.

The restriction of $\pi$ is a morphism of rings with image $R^n$ and
kernel $P_n(A)\cap D$, making \eqref{exactsequence} a short exact sequence of rings. Finally, we have $A\star D=D\star A=\{1\}$, so $P_n(A)\cap D$ is annihilated on both sides by every element of
$A$ under $\star$.
\end{proof}

\begin{corollary}\label{app:power-series}
For every $m\geq1$, we have 
\[
P_m(A)=\{1\}\ \Longrightarrow\ R^m=0
\ \Longrightarrow\ P_{m+1}(A)=\{1\}.
\]
Therefore, if
\[
\nu_A:=\min\{n\geq1:P_n(A)=\{1\}\},\qquad
\nu_R:=\min\{n\geq1:R^n=0\}
\]
are finite, then $\nu_R\leq\nu_A\leq\nu_R+1$.
\end{corollary}
\begin{proof}
If $P_m(A)=\{1\}$, its image $R^m$ is zero. On the other hand, if $R^m=0$,
the exact sequence \eqref{exactsequence} gives $P_m(A)\subseteq D$.
Thus $P_{m+1}(A)=P_m(A)\star A=\{1\}$.
Applying these implications at the least possible indices gives the bounds.
\end{proof}

\begin{remark}
Both bounds can be attained. For a skew brace coming from a nilpotent
radical ring, $D=\{1\}$ and $\nu_A=\nu_R$.
For an almost trivial skew brace on a nonabelian group of nilpotency
class $2$, the $\star$-products are commutators, so
$P_2(A)=D\ne\{1\}$ and $P_3(A)=\{1\}$.
Here $R=A/D$ is a nonzero ring with zero multiplication, giving
$\nu_R=2$ and $\nu_A=3$.
\end{remark}

\begin{corollary}\label{app:finite-abelianization}
If the additive abelianization of $A$ is finite, then $A$ is strongly
$\star$-nilpotent, even if $A$ itself is infinite.
\end{corollary}
\begin{proof}
The radical ring $R=A/D$ is finite and therefore nilpotent, see e.g.\ \cite[Proposition~4.16]{Trappeniers}.
The thesis follows from Corollary~\ref{app:power-series}.
\end{proof}

We next study the lower central series for two-sided skew braces $(A,\cdot,\bullet)$ with $\star$ associative. For subgroups $H,K$ of $A^{\cdot}$, write $[H,K]_{\cdot}$ for the
subgroup generated by their group commutators. Let again $\gamma_1(A^{\cdot})=A$ and
$\gamma_{n+1}(A^{\cdot})=[\gamma_n(A^{\cdot}),A]_{\cdot}$ for $n\geq1$. One defines the lower central series of the skew brace $A$ by
\[
\Gamma_1(A):=A,\quad
\Gamma_{n+1}(A):=\left\langle
[\Gamma_n(A),A]_{\cdot},\,
\Gamma_n(A)\star A,\,A\star\Gamma_n(A)
\right\rangle_{\cdot}\ (n\geq1).
\]
The terms $\Gamma_n(A)$ are ideals of $A$, see e.g.\ \cite[Proposition 7.1]{Tsang}, and this series terminates precisely when the upper central series associated with $\mathrm{Ann}(A)$ reaches $A$ \cite[Theorem 2.7]{BP}, see also with the present indexing \cite[Theorem 7.2]{Tsang}.

\begin{theorem}\label{app:lower-central}
For every $n\geq1$,
\[
\Gamma_n(A)=\gamma_n(A^{\cdot})\cdot P_n(A).
\]
Consequently, $A$ is centrally nilpotent if and only if $A^{\cdot}$
is a nilpotent group and $R$ is a nilpotent ring.
If $A\ne\{1\}$ and these conditions hold, let $c_{\cdot}$ be the
nilpotency class of $A^{\cdot}$ and $c_A$ the central nilpotency class
of the skew brace, with class $c$ meaning that the $(c+1)$st term
is the first trivial term. Then
\[
c_A=\max\{c_{\cdot},\nu_A-1\},
\qquad
\max\{c_{\cdot},\nu_R-1\}\leq c_A
\leq\max\{c_{\cdot},\nu_R\}.
\]
\end{theorem}
\begin{proof}
The formula holds for $n=1$, and
$\Gamma_2(A)=D\cdot P_2(A)$.
Suppose it holds for some $n\geq2$.
Since $\gamma_n(A^{\cdot})\subseteq D$ and
$P_n(A)\subseteq Z(A^{\cdot})$, we have
\[
[\gamma_n(A^{\cdot})\cdot P_n(A),A]_{\cdot}
=\gamma_{n+1}(A^{\cdot}).
\]
The two distributivities \eqref{eq:distrleft} and \eqref{eq:distrright}, together with $D\star A=A\star D=\{1\}$,
also give
\begin{align*}
(\gamma_n(A^{\cdot})\cdot P_n(A))\star A
&=P_n(A)\star A=P_{n+1}(A),\\
A\star(\gamma_n(A^{\cdot})\cdot P_n(A))
&=A\star P_n(A)=P_{n+1}(A).
\end{align*}
This proves the induction step. Each product on the right-hand side
is a subgroup because $\gamma_n(A^{\cdot})$ is normal and $P_n(A)$
is central for $n\geq2$.
Thus $\Gamma_n(A)$ is trivial precisely when both factors are
trivial. 

The equivalence and the class formula now follow from
Corollary~\ref{app:power-series}.
\end{proof}

\begin{remark}
For general skew braces, central nilpotence is equivalent to the
combination of left nilpotence, right nilpotence and nilpotence of the
additive group, see \cite[Theorem~5.3]{Tsang}. For two-sided skew braces
of nilpotent type, equivalence of left, right and strong nilpotence is
already known by \cite[Theorem~4.26]{Trappeniers}. Here
Corollary~\ref{app:power-series} identifies the individual terms of the
three $\star$-series, without assuming that $A^{\cdot}$ is nilpotent.
Theorem \ref{app:lower-central} further separates each lower central
term into its group-commutator and $\star$-product factors and gives the
corresponding class formula.
\end{remark}

\begin{remark}
The group-theoretic hypothesis in Theorem \ref{app:lower-central} is necessary:
on a trivial skew brace, $P_2(A)=\{1\}$, whereas $A^{\cdot}$ can be any group. Thus $\star$-nilpotence alone does not imply central
nilpotence.
\end{remark}

We are also able to characterize two-sided skew brace structures with $\star$-product associative on a group $(G,\cdot)$ which is such that $G/[G,G]\cong\mathbb{Z}$.

\begin{proposition}\label{app:cyclic-abelianization}
Let $(G,\cdot)$ be a group such that $G/[G,G]\cong\mathbb Z$, and fix an
abelianization map $\chi:G\to\mathbb Z$. The two-sided skew brace
structures on $(G,\cdot)$ with associative $\star$-product are exactly
\[
a\bul_z b=a\cdot b\cdot z^{\chi(a)\chi(b)},
\qquad z\in Z(G)\cap[G,G].
\]
Each such structure has $P_3(G)=\{1\}$, and its parameter $z$ is
unique. Moreover, the map
\[
F_z:(G,\cdot)\longrightarrow(G,\bul_z),\ g\mapsto g\cdot z^{\binom{\chi(g)}{2}},
\]
is a group isomorphism, where $\binom{m}{2}=\frac{m(m-1)}{2}$ for $m\in \mathbb{Z}$. In particular, if $Z(G)\cap[G,G]=\{1\}$,
the only such skew brace structure is the trivial one.
\end{proposition}
\begin{proof}
First suppose that $G$ carries such a skew brace structure. The radical
ring $R=G/[G,G]$ has additive group generated by an element $e$ of
infinite order. Write $e^2=ke$, with $k\in\mathbb Z$. For every
$m\in\mathbb Z$, left translation by $me$ in the adjoint group acts
on integer coordinates as
\[
n\longmapsto m+(1+km)n.
\]
It must be bijective. Taking $m=1$ and $m=-1$ gives
$1+k,1-k\in\{1,-1\}$, hence $k=0$. Therefore, we get $e^{2}=0$ and then $R^2=0$.

Choose $t\in G$ with $\chi(t)=1$ so that $e:=tD$ (with $D:=[G,G]$) generates $R^{+}$. Set $z:=t\star t$. By Proposition \ref{app:bilinear}, we obtain
\[
a\star b=\beta(aD,bD)=\beta(\chi(a)e,\chi(b)e)=\beta(tD,tD)^{\chi(a)\chi(b)}=z^{\chi(a)\chi(b)}.
\]
The element $z$ is central, and its image in $R$ is $e^2=0$, so
$z\in[G,G]$. Since $a\bul b=a\cdot(a\star b)\cdot b$ now gives the
claimed formula and proves uniqueness.

Conversely, take $z\in Z(G)\cap[G,G]$, so $\chi(z)=0$.
The formula for $F_z$ defines a bijection, with inverse
$g\mapsto g\cdot z^{-\binom{\chi(g)}{2}}$. 
The identity
\[
\binom{m+n}{2}=\binom m2+\binom n2+mn
\qquad(m,n\in\mathbb Z)
\]
shows directly that $F_z(a\cdot b)=F_z(a)\bul_z F_z(b)$.
Thus $\bul_z$ is a group operation. The lambda actions are given by $\lambda_a(b)=b\cdot z^{\chi(a)\chi(b)}$.
These are automorphisms of $(G,\cdot)$, with inverses obtained by
changing the sign of the exponent, so \eqref{eq:brace-axiom} holds.
The $\star$-product $a\star b=z^{\chi(a)\chi(b)}$ is central and
distributive in the first argument, then Lemma~\ref{lem:conditiontwosided} proves two-sidedness. Since $\chi(z)=0$, both bracketings
of every triple star-product are $1$, proving associativity and
$P_3(G)=\{1\}$.
\end{proof}

\begin{remark}
The construction from a central biadditive map used in
Proposition~\ref{app:cyclic-abelianization} is a special case of
\cite[Theorem~3.2]{CS}. Here the radical-ring quotient and biadditivity
show that these formulas exhaust all the structures under consideration
when the additive abelianization is infinite cyclic.
\end{remark}

A skew brace $(A,\cdot,\bul)$ is called \emph{bi-skew} if
$(A,\bul,\cdot)$ is also a skew brace. We conclude this paper with a short
consequence of the standard lambda-map criteria for a two-sided skew brace $(A,\cdot,\bullet)$ with $\star$ associative.

\begin{corollary}\label{app:bi-skew}
The following conditions are
equivalent:
\begin{enumerate}[label=\textup{\arabic*)}]
\item $A$ is a bi-skew brace;
\item $\lambda:A^{\cdot}\to\Aut(A^{\cdot})$ is a morphism of groups;
\item $P_3(A)=\{1\}$.
\end{enumerate}
In particular, every structure in
Proposition~\ref{app:cyclic-abelianization} is bi-skew.
\end{corollary}

\begin{proof}
From \eqref{eq:distrright} and $A\star A\subseteq Z(A^{\cdot})$ we get
\[
\lambda_{a\cdot b}(c)=((a\cdot b)\star c)\cdot c=(a\star c)\cdot(b\star c)\cdot c=(b\star c)\cdot(a\star c)\cdot c.
\]
Thus, we have 
\[
\lambda_{a}\lambda_{b}(c)=\lambda_{a}(b\star c)\cdot\lambda_{a}(c)=(a\star(b\star c))\cdot(b\star c)\cdot(a\star c)\cdot c=\bigl(a\star(b\star c)\bigr)\cdot\lambda_{a\cdot b}(c).
\]
It follows that $\lambda:A^{\cdot}\to\Aut(A^{\cdot})$ is a morphism of groups if and
only if all triple $\star$-products are $1$. Therefore, 2) and 3) are equivalent.

Since $D\subseteq\Ker(\lambda)\cap\Fix(\lambda)$, we have
$a\cdot d=a\bul d$ and $\lambda_{a\cdot d}=\lambda_a$ for $d\in D$ and $a\in A$. Since $A^{\cdot}/D$ is abelian we know that $(a\cdot b)D=(b\cdot a)D$, hence there exists $d\in D$ such that $a\cdot b=b\cdot a\cdot d$. Thus $\lambda_{a\cdot b}=\lambda_{b\cdot a}$ and 2) and 3) are also equivalent to $\lambda:A^{\cdot}\to\Aut(A^{\cdot})$ is a group antihomomorphism. But this is equivalent to $A$ is bi-skew by \cite[Theorem~2.6]{ST}, so 2) and 3) are equivalent to 1).
\end{proof}

\begin{remark}
For an arbitrary skew brace, the equivalence between the homomorphism
condition on $\lambda$ and $A\star A\subseteq\Ker(\lambda)$ is
\cite[Theorem~3.13]{ST}. The preceding corollary uses the coincidence
of the power series and the annihilation of $D$ to express both that
criterion and the bi-skew criterion by the same condition $P_3(A)=\{1\}$.
\end{remark}

\noindent\textbf{Acknowledgments.} The author is member of the 
GNSAGA-INdAM and is supported by a postdoctoral fellowship at the Université libre de Bruxelles within the framework of the PDR project “Reconstruction of modules and algebraic objects from closed and monoidal structures on their representation categories” funded by the FNRS under the grant number T.0318.25F (PI Joost Vercruysse).

\noindent\textbf{Use of AI}. Claude Opus 5 (Anthropic) was used in an initial exploratory phase and suggested Proposition \ref{prop:quasi}. ChatGPT-6 (OpenAI) was used to finalise the proof of Theorem \ref{thm:extension} and to improve the presentation of the content. The author verified all AI-assisted material and takes full responsibility for the final text.

\end{document}